\documentclass[11pt]{amsart}

\usepackage{amsmath, amssymb,  mathabx, amsfonts,enumerate}
\usepackage[all]{xy}
\usepackage{amscd,stmaryrd}
\usepackage[all]{xy}
\usepackage{comment}
\usepackage{mathtools}
\usepackage{tikz} 
\usepackage{tikz-cd} 
\tikzcdset{diagrams={nodes={inner sep=7.5pt}}}

\usepackage{url}
\usepackage{hyperref}

\newtheorem{theorem}{Theorem}[section]
\newtheorem{lemma}[theorem]{Lemma}
\newtheorem{corollary}[theorem]{Corollary}
\newtheorem{proposition}[theorem]{Proposition}
\newtheorem{theoremletter}{Theorem}

\newtheorem{conjecture}[theorem]{Conjecture}

 \theoremstyle{definition}
 \newtheorem{definition}[theorem]{Definition}

  \newtheorem*{example*}{Example}

\numberwithin{equation}{section}
 
\newcommand {\N}{\mathbb{N}} 
 
\newcommand {\R}{\mathbb{R}} 
 
\newcommand {\C}{\mathbb{C}}

\newcommand{\GG}{\mathcal{G}}

\DeclareMathOperator{\Id}{Id}

\begin{document}
\title[Generalized Gottschalk's conjecture for sofic groups]
{Generalized Gottschalk's conjecture for sofic groups and applications}   
\author[Xuan Kien Phung]{Xuan Kien Phung}
\address{Département de Mathématiques et de Statistique, Université de Montréal, Montréal, Québec, H3T 1J4, Canada.}
\address{Département d'informatique et de recherche opérationnelle,  Université de Montréal, Montréal, Québec, H3T 1J4, Canada.}
\email{phungxuankien1@gmail.com; \quad xuan.kien.phung@umontreal.ca} 

\begin{abstract}
We establish generalizations of the well-known surjunctivity theorem  of Gromov and Weiss as well as the dual-surjunctivity theorem of Capobianco, Kari and Taati for cellular automata (CA) to local perturbations of CA over sofic group universes. We also extend the results to a  class of non-uniform cellular automata (NUCA) consisting of global perturbations with uniformly bounded singularity of CA. 
As an application, we obtain the surjunctivity of algebraic NUCA with uniformly bounded singularity over sofic groups. Moreover, we prove the stable finiteness of twisted group rings over sofic groups to generalize known results on Kaplansky's stable finiteness conjecture for group rings. 
\end{abstract}
\date{\today}
\maketitle
  
\setcounter{tocdepth}{1}

\section{Introduction}  
In computer science, CA is an interesting and powerful model of computation and 
of physical and biological simulation. It is well-known that the CA called Game of Life of Conway \cite{GOL} is Turing complete. Notable recent mathematical theory of CA achieves a  dynamical characterization of amenable groups by the Garden of Eden theorem for CA (cf. \cite{moore}, \cite{myhill}, \cite{ceccherini}, \cite{bartholdi-kielak}) and establishes the equivalence between Kaplansky's stable finiteness conjecture for group rings and the Gottschalk's surjunctivity conjecture for linear CA (cf. \cite{phung-geometric}, \cite{phung-weakly}, \cite{phung-cjm}, \cite{cscp-model}). In this paper, we study the generalized Gottschalk's surjunctivity conjecture for NUCA over the wide class of sofic groups
and obtain several geometric and algebraic applications. 
\par  
\subsection{Preliminaries}
For the main results, we recall briefly notions of symbolic dynamics. 
Given a discrete set $A$ and a group $G$, a \emph{configuration} $x \in A^G$ is  simply a map $x \colon G \to A$.  Two configurations $x,y  \in A^G$ are \emph{asymptotic} if  $x\vert_{G \setminus E}=y\vert_{G \setminus E}$ for some finite subset $E \subset G$.   The \emph{Bernoulli shift} action $G \times A^G \to A^G$ is defined by $(g,x) \mapsto g x$, 
where $(gx)(h) =  x(g^{-1}h)$ for  $g,h \in G$,  $x \in A^G$. The \emph{full shift} $A^G$ is equipped with the prodiscrete topology. For each $x \in A^G$, we define $\Sigma(x) = \overline{\{gx \colon g \in G\}} \subset A^G$ as the smallest closed subshift containing~$x$. 
\par 
Following von Neumann and Ulam  \cite{neumann}, we define a CA over the group $G$ (the \emph{universe}) and the set $A$ (the \emph{alphabet})  as  a $G$-equivariant and uniformly continuous  self-map $A^G \righttoleftarrow$  (cf.~\cite{hedlund-csc}, \cite{hedlund}).  One refers to each element $g \in G$ as a cell of the universe. Then every CA is uniform in the sense that all the cells follow the same local transition map. More generally, when different cells can evolve according to different local transition maps to break down the above  uniformity, we obtain NUCA (cf. \cite[Definition~1.1]{phung-tcs}, \cite{Den-12a}, \cite{Den-12b}):   

\begin{definition}
\label{d:most-general-def-asyn-ca}
Given a group $G$ and an alphabet $A$,  let $M \subset G$ and let $S = A^{A^M}$ be the set of all maps $A^M \to A$. For  $s \in S^G$, we define the NUCA $\sigma_s \colon A^G \to A^G$  by  
$
\sigma_s(x)(g)=  
    s(g)((g^{-1}x)  
	\vert_M)$ for all $x \in A^G$ and $g \in G$. 
 \end{definition} 
\par 
The set $M$ is called a \emph{memory} and $s \in S^G$  the \textit{configuration of local transition maps or local defining maps} of $\sigma_s$.  Every CA is thus a NUCA with finite memory and constant configuration of local defining maps. Conversely, we regard NUCA as \emph{global perturbations} of CA.  When $s, t \in S^G$ are asymptotic,  $\sigma_s$ is a \emph{local perturbation} of  $\sigma_t$ and vice versa. We observe also that NUCA with finite memory are precisely uniformly continuous selfmaps  $A^G \righttoleftarrow$. As for CA, such NUCA satisfy the closed image property \cite[Theorem~4.4]{phung-tcs}, several decidable properties \cite{Den-12a}, \cite{phung-decidable-nuca} 
and  variants of the Garden of Eden theorem \cite{phung-GOE-NUCA}.
\par 
We will analyze in this paper the relations between the following fundamental stable dynamic properties of $\sigma_s$. 
We say that   $\sigma_s$ is \emph{pre-injective} if $\sigma_s(x) = \sigma_s(y)$ implies $x= y$ whenever $x, y \in A^G$ are asymptotic. Similarly, $\sigma_s$ is \emph{post-surjective} if for all $x, y \in A^G$ with $y$ asymptotic to $\sigma_s(x)$,  then   $y= \sigma_s(z)$ for some $z \in A^G$ asymptotic to $x$. We say that $\sigma_s$ is  \emph{stably injective}, resp. \emph{stably post-surjective},  if  $\sigma_p$ is injective, resp. post-surjective,   for every configuration $p \in \Sigma(s)$. 
\par 
The NUCA  $\sigma_s$ is said to be \emph{invertible} if it is bijective and the inverse map  $\sigma_s^{-1}$ is a NUCA with {finite} memory \cite{phung-tcs}. We define also \emph{stable invertibility} which is in general stronger than invertibility, namely, $\sigma_s$ is stably invertible if there exists $N\subset G$ finite and $t \in T^G$ where $T=A^{A^N}$ such that for every $p \in \Sigma(s)$, we have   $\sigma_p \circ \sigma_q=\sigma_q \circ \sigma_p=\Id$  for some $q \in \Sigma(t)$.  In fact, the next lemma shows that every invertible NUCA with finite memory over a finite alphabet and a countable group universe is automatically stably invertible.

\begin{lemma}
\label{l:equivalent-stable-invertible} 
Let $G$ be a countable group and let $M\subset G$ be a finite subset. Let $A$ be a finite alphabet and let $s\in S^G$ where $S= A^{A^M}$. Suppose that $\sigma_s$ is invertible. Then $\sigma_s$ is also stably invertible. 
\end{lemma}

\begin{proof}
Since $\sigma_s$ is invertible, there exists $N \subset G$ finite and $t \in T^G$ where $T=A^{A^N}$ such that $\sigma_s \circ \sigma_t = \sigma_t \circ \sigma_s=\Id$. For every $p \in \Sigma(s)$, we obtain from   the relation $\sigma_s \circ \sigma_t =\Id$ and \cite[Theorem 11.1]{phung-tcs}  some $q \in \Sigma(t)$ such that $\sigma_p \circ \sigma_q=\Id$. Since $q \in \Sigma(t)$ and $\sigma_t \circ \sigma_s=\Id $, another application of \cite[Theorem 11.1]{phung-tcs} shows that there exists $r \in \Sigma(s)$ such that $\sigma_q \circ \sigma_r=\Id$. It follows that $\sigma_q$ is invertible and thus $\sigma_p \circ \sigma_q=\sigma_q \circ \sigma_p=\Id$. We conclude that $\sigma_s$ is stably invertible. 
\end{proof}

\par 
Note that for CA, stable invertibility, resp. stable injectivity, resp. stable post-surjectivity, is equivalent to invertibility, resp.   injectivity, resp. post-surjectivity since $\Sigma(s)=\{s\}$ if $s \in S^G$ is constant. 

\subsection{Main dynamical results}

\subsubsection{Stable surjunctivity for sofic groups}
The surjunctivity conjecture of Gottschalk \cite{gottschalk}  asserts that over any group universe, every injective CA with finite alphabet must be surjective. 
Over sofic group universes (Section  \ref{s:sofic-groups}), the surjunctivity conjecture was famously shown by Gromov and Weiss in  \cite{gromov-esav},  \cite{weiss-sgds}. The vast class of sofic groups was first  introduced by Gromov  \cite{gromov-esav} and includes all amenable groups and all residually finite groups. The question of whether there exists a non-sofic group is still open. 
\par 
 The situation for the surjunctivity of NUCA is more subtle. While stably injective NUCA with finite memory may fail to be  surjective (see \cite[Example 14.3]{phung-tcs}), results in \cite{phung-tcs} show that stably injective local perturbations of CA with finite memory over an amenable group or an residually finite group universe must be surjective. 
More generally, we establish the following generalization of the Gromov-Weiss surjunctivity theorem to cover all {stably injective} local perturbations of CA over sofic group universes. 
 
\par 
\begin{theoremletter}
\label{t:main-A}
Let $M$  be a finite subset of a finitely generated  sofic group $G$. Let $A$ be a finite  alphabet and  $S=A^{A^M}$. Suppose that $\sigma_s$ is stably injective for some asymptotically constant $s \in S^G$. Then $\sigma_s$ is stably invertible.  
\end{theoremletter}
\par 
Combining with the result in \cite{phung-tcs} where we can replace the stable injectivity by the weaker injectivity condition whenever $G$ is an amenable group, we obtain the following general surjunctivity and invertibility result for NUCA which are local perturbations of classical CA. 
\par 
\begin{corollary}
  Let $M$  be a finite subset of a group $G$. Let $A$ be a finite alphabet and  $S=A^{A^M}$. Let $s \in S^G$ be an asymptotic constant configuration such that $\sigma_s$ is injective. Then $\sigma_s$ is stably invertible in the following cases: 
  \begin{enumerate}[\rm (i)] 
      \item $G$ is an amenable group;
      \item 
      $G$ is a sofic group and $\sigma_c$ is injective. 
  \end{enumerate}
\end{corollary}
 
 \par 
 To deal with global perturbations of CA, we introduce the following notion of uniformly bounded singularity. 
\begin{definition} A NUCA $\sigma_s \colon A^G \to A^G$ has \emph{uniformly bounded singularity} if for every  $E\subset G$ finite with $1_G\in E=E^{-1}$, we can find a finite subset $F\subset G$ containing $E$  such that the restriction  $s\vert_{FE\setminus F}$ is constant. In this case, we also say that the configuration $s$ has \emph{uniformly bounded singularity}. 
\end{definition}
It is clear that $\sigma_s$ has uniformly bounded singularity if  $s$ is asymptotically constant. When the universe $G$ is a residually finite group, our notion of uniformly bounded singularity is closely related but not equivalent to the notion of (periodically) bounded singularity for NUCA defined in   \cite[Definition~10.1]{phung-tcs}. 
The next result confirms that  Theorem~\ref{t:main-A} can be generalized to NUCA with uniformly bounded singularity. 
\begin{theoremletter}
    \label{t:main-B} 
 Let $M$  be a finite subset of a finitely generated  sofic group $G$. Let $A$ be a finite alphabet and let $S=A^{A^M}$. Suppose that $\sigma_s \colon A^G \to A^G$ is stably injective for some $s \in S^G$ with uniformly bounded singularity.  Then $\sigma_s$ is stably invertible.     
\end{theoremletter}
In particular, our result extends \cite[Theorem~C]{phung-tcs} to cover the class of NUCA with uniformly bounded singularity over sofic group universes. Note again that \cite[Example 14.3]{phung-tcs} implies that we cannot remove the uniformly singularity condition in the above theorem.

\subsubsection{Stable dual-surjunctivity for sofic groups}
Our next  results concern the dual-surjunctivity  version of Gottschalk's conjecture was introduced recently by 
Capobianco, Kari, and Taati in \cite{kari-post-surjective} which  states that   every post-surjective CA over a group universe and a finite alphabet is also pre-injective. Moreover, the authors settled in the same paper \cite{kari-post-surjective} the dual-surjunctivity conjecture for CA over sofic universes. See also \cite{phung-post-surjective} for some extensions.  
\par 
We establish the following  extension of the dual-surjunctivity theorem of Capobianco, Kari, and Taati to the class of post-surjective local perturbations of CA over sofic groups.  

\begin{theoremletter}
\label{t:main-C} 
Let $G$ be a finitely generated sofic group and let $A$ be a finite alphabet.  Let $M\subset G$ be finite and $S=A^{A^M}$. Let $s \in S^G$ be  asymptotically constant such that $\sigma_s$ is stably post-surjective. Then $\sigma_s$ is stably invertible. 
\end{theoremletter}
\par 
Given a vector space $V$ which is not necessarily finite, $V^G$ is  a vector space with component-wise operations and a NUCA $\tau \colon V^G \to V^G$ is said to be \emph{linear} if it is also a linear map of vector spaces or equivalently, if its local transition maps are all linear. Linear NUCA with finite memory enjoy similar properties as linear CA such as the shadowing property  \cite{birkhoff-pseudo-orbit}, \cite{bowen-shadow-equilibrium}, \cite{kurka-97}, \cite{blanchard-maass-97},  \cite{meyerovitch-pseudo-orbit}, \cite{chung-shadow}, \cite{phung-israel}, \cite{phung-shadowing}, \cite{cscp-jpaa}, \cite{phung-laa}. 
The recent result \cite[Theorem D.(ii)]{phung-laa}  states that every stably post-surjective linear NUCA with finite vector space alphabet  over a residually finite group universe is invertible whenever it is a local perturbation of a linear CA. When specialized to the class of linear NUCA,   Theorem~\ref{t:main-C} thus generalizes \cite[Theorem D.(ii)]{phung-laa} because every residually finite group is sofic. 
As for Theorem~\ref{t:main-B}, we obtain the following extension of Theorem~\ref{t:main-C} to cover the class of NUCA with uniformly bounded singularity. 
\begin{theoremletter}
\label{t:main-D} 
 Let $G$ be a finitely generated sofic group and let $A$ be a finite alphabet.  Let $M\subset G$ be finite and let $S=A^{A^M}$. Suppose that $\sigma_s \colon A^G \to A^G$ is stably post-surjective for some $s \in S^G$ with uniformly bounded singularity. Then $\sigma_s$ is stably invertible. 
\end{theoremletter}

\subsubsection{Conjectures}
Our main results naturally motivate the following surjunctivity conjectures for NUCA with uniformly bounded singularity which generalizes the surjunctivity and the dual surjunctivity conjectures for CA. 

\begin{conjecture}[Stable surjunctivity] 
\label{conjecture-stable-surj} Let $G$ be a finitely generated group and let $A$ be a finite alphabet.  Let $M\subset G$ be finite and let $S=A^{A^M}$. Suppose that $\sigma_s \colon A^G \to A^G$ is stably injective for some $s \in S^G$ with uniformly bounded singularity. Then $\sigma_s$ is stably invertible. 
\end{conjecture}

\begin{conjecture}[Stable dual-surjunctivity] Let $G$ be a finitely generated group and let $A$ be a finite alphabet.  Let $M\subset G$ be finite and let $S=A^{A^M}$. Suppose that $\sigma_s \colon A^G \to A^G$ is stably post-surjective for some $s \in S^G$ with uniformly bounded singularity. Then $\sigma_s$ is stably  invertible. 
\end{conjecture}

By Theorem~\ref{t:main-B} and Theorem~\ref{t:main-D}, the above conjectures hold true for finitely generated sofic groups. We observe that by results in \cite{phung-laa} and Lemma~\ref{l:equivalent-stable-invertible}, the stable surjunctivity conjecture and the stable dual-sujunctivity conjecture are in  fact equivalent when restricted to the class of linear NUCA. 

\begin{proposition}
\label{pro:intro}
Let $G$ be a finitely generated group and let $A$ be a finite vector space alphabet. 
The stable surjunctivity conjecture holds true for all linear $\mathrm{NUCA}$  $A^G\to A^G$  of finite memory with uniformly bounded singularity if and only if so does the stable dual-surjunctivity conjecture. 
\end{proposition}

\begin{proof}
Let $M\subset G$ be a finite subset and let $s\in S^G$ where $S=A^{A^M}$. We identify the vector space $A$ with its dual vector space $A^*$. Then we infer from \cite{phung-laa} that $\sigma_s$ admits some dual NUCA $\sigma_{s^*}\colon A^G \to A^G$ with finite memory where  $s^*\in S^G$ such that 
\begin{enumerate}[\rm (i)]
    \item $\sigma_s$ is invertible$\iff$$\sigma_{s^*}$ is invertible,
    \item 
    $\sigma_s$ is stably injective$\iff$$\sigma_{s^*}$ is stably post surjective,
    \item 
     $\sigma_s$ is stably post-surjective$\iff$$\sigma_{s^*}$ is stably injective. 
\end{enumerate}
 Moreover, we infer from Lemma~\ref{l:equivalent-stable-invertible} that every invertible NUCA with finite memory over a finite alphabet and a countable group universe is automatically stably invertible.   
 Finally, it is straightforward  from the definition of $s^*$ that $s$ has uniformly bounded singularity if and only if so does $s^*$. The proof is thus complete.  
\end{proof}

\subsection{Applications} 

We will present below several algebraic and geometric  applications of our main results.

\subsubsection{Direct finiteness of algebraic NUCA} 
Our first application is the following  geometric result which generalizes the direct finiteness property of algebraic CA over sofic groups   \cite{phung-geometric} to the class of algebraic NUCA (see e.g. \cite{gromov-esav}, \cite{cscp-comalg}, \cite{cscp-alg-goe}, \cite{cscp-invariant}, \cite{phung-embedding}) with uniformly bounded singularity over sofic groups. In this paper, an algebraic variety is a reduced separated scheme of finite type over an algebraically closed field and we  identify such a variety with its closed points. 

\begin{theoremletter}
\label{t:main-E} 
Let $G$ be a finitely generated sofic group and let $X$ be an algebraic variety. Let $M\subset G$ be finite and let $S$ be the set of morphisms of algebraic varieties $X^M \to X$. Suppose that $s, t\in S^G$ have uniformly bounded singularity and  $\sigma_s \circ \sigma_t=\Id_{X^G}$. Then $\sigma_t \circ \sigma_s=\Id_{X^G}$. 
\end{theoremletter}

\begin{proof}
By using Lemma~\ref{l:main-lemma-singular}, the proof follows closely, \emph{mutatis mutandis}, the proof of \cite[Theorem~C]{phung-geometric} which reduces the surjunctivity of $\sigma_s$ to the surjunctivity of \emph{all} NUCA $\sigma_q\colon A^G \to A^G$ where $A$ is a finite alphabet,  $N\subset G$ is a finite subset, $Q=A^{A^N}$, and $q \in Q^G$ is a configuration with uniformly bounded singularity. 
Consequently,  Theorem~\ref{t:main-E} from  Theorem~\ref{t:main-B} by the above reduction. The proof is thus complete. 
\end{proof}

We remark that the same proof above actually shows that Theorem~\ref{t:main-E} still holds true when we replace the group $G$ by any stable surjunctive group, namely, a group which satisfies  Conjecture~\ref{conjecture-stable-surj} on the stable surjunctivity.

\subsubsection{Stable finiteness of twisted group rings}
Our next application contributes to 
Kaplansky's stable finiteness conjecture \cite{kap} which states that for every group $G$ and every field $k$, the group ring $k[G]$ is stably finite, that is, 
every one-sided invertible element of the ring $M_n(k[G])$ of square matrices of size $n \times n$  with coefficients in $k[G]$ must be a two-sided unit. 
 Kaplansky's stable finiteness holds true for all sofic groups  \cite{ara}, \cite{elek}, \cite{csc-sofic-linear} and  when $k=\C$ regardless of the group $G$ \cite{kap}, \cite{juschenko}. 

\par 
Given a group $G$ and a vector space $V$ over a field $k$, the $k$-algebra   $\mathrm{LNUCA}_{c}(G,V)$ consists of linear NUCA $V^G \to V^G$ of finite memory which are local perturbations of a linear CA.  
The multiplication in  $\mathrm{LNUCA}_{c}(G, V)$ is induced by the composition of maps and the addition is component-wise. 
\par 
 By results in \cite{csc-sofic-linear}, there exists an isomorphism  $M_n(k[G])\simeq \mathrm{LCA}(G, k^n)$ of $k$-algebras where $\mathrm{LCA}(G, k^n) \subset \mathrm{LNUCA}_{c}(G, k^n)$ is the subalgebra consisting of linear CA. When $G$ is infinite,  we can extend the above isomorphism to an isomorphism  
 $M_n(D^1(k[G])) \simeq \mathrm{LNUCA}_{c}(G, k^n)$ (see \cite{phung-cjm}) where   $$D^1(k[G]) =   k[G] \times (k[G])[G]$$ is the twisted group ring with component-wise addition and with multiplication given by
$ 
(\alpha_1, \beta_1) * (\alpha_2, \beta_2) = (\alpha_1 \alpha_2, \alpha_1 \beta_2 + \beta_1 \alpha_2 + \beta_1 \beta_2)$. Here,  $\alpha_1 \alpha_2$ is computed with the multiplication rule in the group ring $k[G]$ so that $k[G]$ is naturally a subring of $D^1(k[G])$ via the map $\alpha \mapsto (\alpha, 0)$. However,  $\alpha_1 \beta_2$,  $\beta_1 \alpha_2$, $\beta_1 \beta_2$ are {twisted products} and defined for $g, h\in G$ by $
    (\alpha \beta)(g)(h) = \sum_{t \in G} \alpha(t) \beta(gt)(t^{-1}h)$, $
    (\beta \alpha)(g)(h)= \sum_{t \in G}  \beta(g)(t) \alpha(t^{-1}h)$, $
    (\beta\gamma)(g)(h)= \sum_{t \in G} \beta(g)(t) \gamma(gt)(t^{-1}h)$. 
Twisted products  are  different from the multiplication rule of the group ring $(k[G])[G]$  with coefficients in $k[G]$.  

\par 
It was shown  that all sofic groups are {$L$-surjunctive} \cite{gromov-esav}, \cite{csc-sofic-linear}, that is, for every finite-dimensional vector space $V$,   every injective $\tau  \in \mathrm{LCA}(G,V)$ is also surjective.  Moreover,  a group $G$ is $L$-surjunctive if and only if $k[G]$ is stably finite for every field $k$ \cite{csc-sofic-linear}.  
In general, we say that a group  $G$ is {$L^1$-surjunctive}, resp.  {dual $L^1$-surjunctive}, if for every finite-dimensional vector space $V$ and $\tau  \in \mathrm{LNUCA}_c(G,V)$ that is stably  injective, resp. stably post-surjective, then $\tau$ must be surjective, resp. pre-injective.  
\cite[Theorem~C]{phung-cjm} shows that  initially subamenable groups  and  residually finite groups are both $L^1$-surjunctive and dual $L^1$-surjunctive. As an application of Theorem~\ref{t:main-A}, we obtain a  generalization of \cite[Theorem~C]{phung-cjm} which also confirms the stable finiteness of twisted  group rings over sofic groups. 
    
\begin{theoremletter}
\label{t:main-F} 
    Let  $G$ be a finitely generated sofic group. Then we have  
\begin{enumerate}[\rm (i)]
\item 
   $G$ is $L^1$-surjunctive,
\item 
   $G$ is dual $L^1$-surjunctive, 
\item $D^1(k[G])$ is stably finite for every field $k$. 
\end{enumerate}
\end{theoremletter}

\begin{proof}
The theorem is trivial if $G$ is finite so we can suppose that $G$ is infinite. We infer from Theorem~\ref{t:main-A} that for every \emph{finite} vector space $V$, every stably injective $\tau \in \mathrm{LNUCA}_c(G,V)$  must be invertible. Since invertibility implies surjectivity, we deduce that $G$ is   \emph{finitely} $L^1$-surjunctive (see \cite[Definition~1.3]{phung-cjm}). Consequently, we conclude from  \cite[Theorem~B]{phung-cjm} that $G$ is both $L^1$-surjunctive and  dual $L^1$-surjunctive and that $D^1(k[G])$ is stably finite for every field $k$. The proof is thus complete. 
\end{proof}

\subsection{Organization of the paper}  

We concisely collect  the definition and  basic properties of finitely generated sofic groups in  Section~\ref{s:sofic-groups}. We then fix the notations and recall the construction of induced local maps of NUCA in Section~\ref{s:induced-local-map}. The proof of Theorem~\ref{t:main-A} is given in Section~\ref{s:local-gottschalk-sofic} where we need to relate several stable dynamic properties of NUCA (Lemma~\ref{l:reversible-asymp-constant}) and analyze the global-local structures of stably injective NUCA with uniformly bounded singularity in the proof of Theorem~\ref{t:main-A-proof}. Then in Section~\ref{s:dual-surjunctive-sofic}, we give the proof of Theorem~\ref{t:main-C} and Theorem~\ref{t:main-D} which is similar to but less involved than the proof of Theorem~\ref{t:main-A}. Finally, the proof of Theorem~\ref{t:main-B} is given in Section~\ref{s:surjunctivity-sofic-general} which reduces Theorem~\ref{t:main-B} to Theorem~\ref{t:main-A} by a technical lemma (Lemma~\ref{l:main-lemma-singular}).

\section{Sofic groups}
\label{s:sofic-groups}

The important class of sofic groups was introduced by Gromov \cite{gromov-esav}  as a common generalization of residually finite groups and  amenable groups. 
Many conjectures for groups have been established for the sofic ones such as 
Gottschalk's surjunctivity conjecture and Kaplansky's stable finiteness conjecture   \cite{gromov-esav}, \cite{weiss-sgds}, \cite{elek},  \cite{csc-book}, \cite{phung-geometric}, \cite{phung-weakly}. We recall below a useful geometric characterization of finitely generated sofic groups. 
 \par
Given a finite set $\Delta$, a   $\Delta$-labeled graph is a pair $\GG= (V,E)$, 
where $V$ is the set of {vertices}, 
and $E \subset V \times \Delta \times V$ is the set of $\Delta$-labeled edges. 
The length of a path $\rho$ in $\GG$ is denoted by $l(\rho)$.  
For $v, w\in V$, we set   $d_\GG(v,w)= \min \{ l(\rho): \text{$\rho$ is a path from $v$ to $w$}\}$ if  $v$ and $w$  are connected by a path, and 
$d_\GG(v,w)= \infty$ 
otherwise. 
For $v\in V$ and $r \geq 0$, we have a $\Delta$-labeled subgraph $B_\GG(v,r)$ of $\GG$ defined by: 
\[
B_\GG(v,r)= \{w\in V: d_\GG (v,w) \leq r\}. 
\]
\par
Let $(V,E)$ and $(W,F)$ be $\Delta$-labeled graphs. 
A map $\phi \colon V \to W$ is  an $\Delta$-labeled graph homomorphism 
from $(V, E)$ to $(W,F)$ if $(\phi(v),s, \phi(w)) \in F$ for all $(v,s,w)\in E$. 
A bijective $\Delta$-labeled graph homomorphism $\phi \colon V \to W$ is an 
$\Delta$-labeled graph isomorphism if its inverse  $\phi^{-1}\colon W \to V$ 
is an  $\Delta$-labeled graph homomorphism. 
\par
Let $G$ be a finitely generated group and let $\Delta\subset G$ 
be a finite symmetric generating subset, i.e., $\Delta=\Delta^{-1}$. 
The Cayley graph of $G$ with respect to $\Delta$ is the connected $\Delta$-labeled graph $C_\Delta(G) = (V,E)$,  
where $V = G$ and $E=\{(g,s,gs): g\in G \text{ and } s \in \Delta)\}$.  
For $g \in G$ and $r\geq 0$, we denote 
\[
B_\Delta(r)= B_{C_\Delta(G)}(1_G,r)  
\]  
\par 
\noindent
We have the following characterization of sofic groups (\cite[Theorem~7.7.1]{csc-book}).

\begin{theorem}
\label{t:sofic-character}
Let $G$ be a finitely generated group. 
Let $\Delta\subset G$ be a finite symmetric generating subset. 
Then the following are equivalent: 
\begin{enumerate} [\rm (a)]
\item
the group $G$ is sofic;
\item
for all $r, \varepsilon >0$, there exists a finite $\Delta$-labeled graph $\GG=(V,E) $ 
satisfying 
\begin{equation*} 
\vert V(r) \vert \geq (1 - \varepsilon) \vert V \vert,
\end{equation*}
where $V(r)\subset V$ consists of $v \in V$ such that there exists a unique $S$-labeled graph
isomorphism $\psi_{v,r} \colon  B_\GG(v,r)\to B_\Delta(r)$ with $\psi_{v,r}(v) = 1_G$. 
\end{enumerate}
\end{theorem}
\noindent
If $0 \leq r\leq s$ then  $V(s) \subset V(r)$ as every $S$-labeled graph isomorphism $\psi_{v,s} \colon  B_\GG(v,s) \to B_\Delta(s)$ induces by restriction  
an $\Delta$-labeled graph isomorphism $ B_\GG(v,r) \to B_\Delta(r)$. 
We will also need the  following Packing lemma (cf. \cite{weiss-sgds},  \cite[Lemma~7.7.2]{csc-book}, see also \cite{phung-2020} for (ii)).  

\begin{lemma}
\label{l:sofic-B-V}
With the notation as in Theorem~\ref{t:sofic-character}, the following hold 
\begin{enumerate}[\rm(i)]
\item
$B_\GG(v,r) \subset V(kr)$ for all $v \in V((k+1)r)$ and $k \geq 0$;
\item
There exists a finite subset $W \subset V(3r)$ such that the balls $B_\GG(w,r)$ 
are pairwise disjoint for all $w\in W$ and that 
$V(3r) \subset \bigcup_{w\in W} B_\GG(w,2r)$. 
\end{enumerate}
\end{lemma}

\section{Induced local  maps of NUCA} 
\label{s:induced-local-map}
Let $G$ be a group and let $A$ be an alphabet. For every subset $E\subset G$ and $x \in A^E$ we define  $gx \in A^{gE}$ by  $gx(gh)=x(h)$ for all $h \in E$. In particular,    $gA^E=A^{gE}$. 
Let $M\subset G$ and  let $S=A^{A^M}$ be the collection of all maps $A^M \to A$. 
For every finite subset $E \subset G$  and $w \in S^{E}$,  
we define a map  $f_{E,w}^{+M} \colon A^{E M} \to A^{E}$  as follows. For every $x \in A^{EM}$ and $g \in E$, we set: 
\begin{align}
\label{e:induced-local-maps} 
    f_{E,w}^{+M}(x)(g) & = w(g)((g^{-1}x)\vert_M). 
\end{align}
\par 
\noindent 
In the above formula, note that  $g^{-1}x \in A^{g^{-1}EM}$ and $M \subset g^{-1}EM$ since $1_G \in g^{-1}E$ for $g \in E$. Therefore, the map   $f_{E,w}^{+M} \colon A^{E M} \to A^{E}$ is well defined. 
Consequently, for every $s \in S^G$, we have a well-defined induced local map $f_{E, s\vert_E}^{+M} \colon A^{E M} \to A^{E}$ for every finite subset $E \subset G$ which satisfies: 
\begin{equation}
\label{e:induced-local-maps-general} 
    \sigma_s(x)(g) =  f_{E, s\vert_E}^{+M}(x\vert_{EM})(g)
\end{equation}
for every $x \in A^G$ and $g \in E$. Equivalently, we have for all $x \in A^G$ that: 
\begin{equation}
\label{e:induced-local-maps-proof} 
    \sigma_s(x)\vert_E =  f_{E, s\vert_E}^{+M}(x\vert_{EM}). 
\end{equation}

For every $g \in G$, we have a canonical  bijection  $\gamma_g\colon G \mapsto G$ induced by the translation $a \mapsto g^{-1}a$. For each subset $K \subset G$, we denote by $\gamma_{g, K} \colon gK \to K$ the restriction to $gK$ of $\gamma_g$. 
Now let $N\subset G$ and $T=A^{A^N}$. Let $t \in T^G$. With the above notations, we have the following auxiliary lemma.  

\begin{lemma}
    \label{l:compo-id} Suppose that $1_G\in M\cap N$. Then for every $g \in G$, the condition $\sigma_t(\sigma_s(x))(g)=x(g)$ for all $x \in A^G$ is equivalent to the condition 
    $$t(g)\circ  \gamma_{g, N} \circ f^{+M}_{gN, s\vert_{gN}}\circ \gamma^{-1}_{g, NM}=\pi,$$
    where $\pi\colon A^{NM} \to A$ is the projection $z\mapsto z(1_G)$. 
\end{lemma}

\begin{proof}
  For every $g \in G$ and $x \in A^G$, we deduce from Definition~\ref{d:most-general-def-asyn-ca} and the relation \eqref{e:induced-local-maps-proof} that 
  \begin{align*}
     \sigma_t(\sigma_s(x))(g) & = t(g) \left( (g^{-1}\sigma_s(x))\vert_N\right)\\
     & = t(g) \left( \gamma_{g,N}\left( (\sigma_s(x))\vert_{gN}\right)\right)\\
     & =  t(g) \left( \gamma_{g,N}\circ f^{+M}_{gN, s\vert_{gN}}(x\vert_{gNM})\right) \\
     & =  t(g) \circ \gamma_{g,N}\circ f^{+M}_{gN, s\vert_{gN}}(x\vert_{gNM}) \\
     & =  t(g) \circ  \gamma_{g,N}\circ f^{+M}_{gN, s\vert_{gN}}\circ \gamma^{-1}_{g, NM}\left((g^{-1}x)\vert_{NM}\right) 
  \end{align*}
  whence the conclusion as $x(g)= (g^{-1}x)(1_G)$. 
\end{proof}

\section{Invertibility of stably injective local perturbations of CA} 
\label{s:local-gottschalk-sofic}
For the proof of Theorem~\ref{t:main-A}, we  need the following auxiliary result.  
\begin{lemma}
\label{l:reversible-asymp-constant}
Let $M$  be a finite subset of a group $G$. Let $A$ be a finite alphabet and let $S=A^{A^M}$. Suppose that $\sigma_s \colon A^G \to A^G$ is stably injective for some asymptotically constant $s \in S^G$. Then there exists a finite subset $N\subset G$ and an asymptotically constant configuration $t \in T^{G}$ where $T=A^{A^N}$ such that  $\sigma_t \circ \sigma_s=\Id$.     
\end{lemma}

\begin{proof}
As $s\in S^G$ is asymptotically constant, we can find a finite subset $E \subset G$ and a constant configuration $c \in S^G$ such that $s\vert_{G\setminus E} = c\vert_{G\setminus E}$. 
Since $\sigma_s$ is stably injective, we infer from \cite{phung-tcs} that $\sigma_s$ is  reversible, that is, there exists a finite subset $N \subset G$ and 
 $r \in T^G$ where $T=A^{A^N}$ such that 
 $\sigma_r\circ \sigma_s =\Id$. Up to enlarging $M,N$ if necessary, we can suppose without loss of generality that $N=N^{-1}$ and $1_G \in  M \cap N$. Hence, $1_G \in N^2$ and we obtain a canonical projection map $\pi \colon A^{N^2}\to A$ given by $x\mapsto x(1_G)$. 
 For $g \in G \setminus EN$, we have $gN \cap E= \varnothing$ and thus $s\vert_{gN}= c\vert_{gN}$. Consequently, for every $g \in G \setminus EN$, the condition $\sigma_r( \sigma_s(x))(g)=x(g)$ for all $x \in A^G$   is equivalent to the condition $r(g) \circ f^{+N}_{N, c\vert_N} = \pi $ by Lemma~\ref{l:compo-id}. We fix $g_0 \in G \setminus EN$ and 
 define an asymptotically constant configuration $t \in T^G$ by setting $t(g)=r(g_0)$ if $g \in G\setminus EN$ and $t(g)= r(g)$ if $g \in EN$. It follows from our construction that $\sigma_t( \sigma_s(x))(g)=x(g)$ for all $x \in A^G$ and $g \in G \setminus EN$. As $t\vert_{EN}=r\vert_{EN}$, we infer from $\sigma_r\circ \sigma_s=\Id$ that $\sigma_t( \sigma_s(x))(g)=x(g)$ for all $x \in A^G$ and  $g \in  EN$.  
Therefore,  $\sigma_t\circ \sigma_s=\Id$ and the proof is complete. 
\end{proof}
We are now in position to give the proof of the following result which together with Lemma~\ref{l:equivalent-stable-invertible} prove Theorem~\ref{t:main-A}. 

\begin{theorem}
\label{t:main-A-proof}
Let $M$  be a finite subset of a finitely generated  sofic group $G$. Let $A$ be a finite set and let $S=A^{A^M}$. Suppose that $\sigma_s \colon A^G \to A^G$ is stably injective for some asymptotically constant $s \in S^G$. Then $\sigma_s$ is invertible.  
\end{theorem}

\begin{proof} 
Let us denote $\Gamma = \sigma_s(A^G) \subset A^G$. 
Let $\Delta \subset G$ be a finite symmetric generating subset of $G$ such that $1_G \in \Delta$.  
For $n \in \N$, we denote by  $B_\Delta(n)$ 
the   ball  centered at 
$1_G$ of radius $s$ in the Cayley graph $C_\Delta(G)$. 
\par
Since $\sigma_s$ is stably injective and $s\in S^G$ is asymptotically constant, we infer from Lemma~\ref{l:reversible-asymp-constant} that  
there exists a nonempty finite subset $N\subset G$ 
and an asymptotically constant configuration $t \in T^G$, where $T=A^{A^N}$, such that $\sigma_t\circ \sigma_s=\Id$.  We can thus find constant configurations $c \in S^G$ and $d \in T^G$ and a finite subset $F \subset G$ such that $s\vert_{G \setminus F} =  c\vert_{G \setminus F}$ and $t\vert_{G \setminus F} =  d\vert_{G \setminus F}$.  
\par
We suppose on the contrary that $\sigma_s$ is not surjective, i.e., $\Gamma \subsetneq A^G$. 
It follows from \cite[Theorem~4.4]{phung-tcs} that 
$\Gamma$ is closed in $A^G$ with respect to the prodiscrete topology.  
Hence, there exists a finite subset $E \subset G$ 
such that $\Gamma\vert_E \subsetneq A^E$ since otherwise, $\Gamma$ would be a dense closed subset of $A^G$ and thus  $\Gamma=A^G$ which is a contradiction. 
\par 
We fix $r \in \N$ large enough such that $r \geq 1$ and   
$B_\Delta(r) \supset M \cup N \cup E\cup F$.   
Up to enlarging $M$, $N$, $E$, and $F$ if necessary,  we can suppose without loss of generality that 
$$M=N=E=F=B_\Delta(r)$$ and thus in particular $S=T$ and 
\begin{equation}
\label{e:contrary}
\Gamma \vert_{B_S(r)} \subsetneq A^{B_S(r)}. 
\end{equation}
\par
Let $R=4r$. If $\vert A \vert \in \{0,1\}$ then the theorem is trivial. Hence, we   suppose in the rest of the proof that $\vert A \vert \geq 2$. Therefore, $\log \left( 1- \vert A \vert^{-\vert B_\Delta(R) \vert} \right) <0$ and we can thus fix $\varepsilon \in \R $ which satisfies  
\begin{equation} \label{e:contra-hypothesis} 
   0 < \varepsilon < 1 - \frac{\vert B_\Delta(2R) \vert \log \vert A \vert}{\vert B_\Delta(2R) \vert \log \vert A \vert - \log \left( 1- \vert A \vert^{-\vert B_\Delta(R) \vert} \right)}.  
\end{equation}
\par
Since the group $G$ is sofic,   Theorem~\ref{t:sofic-character} implies 
that there exists a finite $S$-labeled graph $\GG=(V,E)$ associated to the pair $(3R, \varepsilon)$ 
such that
\begin{equation} 
\label{e:sofic-V}
    \vert V(3R) \vert \geq (1 - \varepsilon) \vert V \vert,
\end{equation}
where for every $n \in \N$, the subset $V(n)\subset V$ 
consists of all $v \in V$ such that there exists a  $\Delta$-labeled graph
isomorphism $\psi_{v,n} \colon B_\Delta (n) \to B(v,n)$ which satisfies  
$\psi_{v,n}(1_G) = v$ (cf.~Theorem \ref{t:sofic-character}) where we denote $B(v,n)= B_\GG(v,n)$ for all $v \in V$ and $n\in \N$. Thus, for every $g \in G$, we obtain the following  $\Delta$-labeled graph isomorphism where we recall that  $\gamma_{g, K} \colon gK \to K$, for every subset $K \subset G$, is the restriction to $gK$ of the translation map $a \mapsto g^{-1}a$: 
$$ \psi_{v,n}\circ \gamma_{g, B_\Delta(n)} \colon gB_\Delta(n) \to B(v,n).$$ 
\par 
\noindent 
Note 
that $V(m) \subset V(n)\subset V$ for all $m \geq n \geq 0$. 
We obtain from Lemma~\ref{l:sofic-B-V}.(ii) a subset $W \subset V(3R) $ such that 
$B(w,R)$ are pairwise disjoint for all $w\in W$ and that $V(3R) \subset \bigcup_{w \in W}B(w,2R)$. Since $W\subset V(3R)$, we have $|B(w,2R)|=|B_\Delta(2R)|$ for every $w \in W$ and thus 
\begin{align}
    \label{eq:cardinality-w-main-A}
    |W||B_\Delta(2R)| \geq \left| \bigcup_{w \in W}B(w,2R) \right| \geq |V(3R)|. 
\end{align}
\par 
\noindent
Let us denote 
$\overline{W} = \coprod_{w \in W} B(w,R)$ then 
  $\overline{W} \subset V(2R)$  by Lemma~\ref{l:sofic-B-V}.(i). Since $|B(w,R)|=|B_\Delta(R)|$ for all $w \in W$, we find that 
\begin{equation} 
\label{e:sofic-Vr}
\vert V(2R)\vert = \vert W \vert \vert B_\Delta(R)\vert + \vert  
V(2R) \setminus \overline{W}\vert.
\end{equation}
\par 
\noindent 
We construct a  map $\Phi\colon    A^{V(R)} \to A^{V(2R)}$ defined by the following formula for all $x \in A^{V(R)}$ and $v \in V(2R)$ (note that $B(v,r) \subset B(v,R) \subset V(R)$ by Lemma~\ref{l:sofic-B-V}.(i)):  
\begin{align}
\label{eq:Phi-map}
\Phi(x)(v) 
=
\begin{cases}
    s(\psi^{-1}_{w,R}(v)) (x\vert_{B(v,r)}\circ \psi_{v,r}) & \text{ if } v\in B(w,R) \text{  for some } w \in W\\
    c(1_G)\left(x\vert_{B(v,r)} \circ \psi_{v,r}\right) & \text{ if } v \in V(2R) \setminus \overline{W}.  
\end{cases}
\end{align}
By the choice of $B_\Delta(r)=M=F$ and $R=4r$, observe that for every  $u \in B(v, r)\cap B(w,R)$ where $v \in V(2R) \setminus \overline{W}$ and $w \in W$, we have 
\begin{align}
    \label{eq:definition-Phi-compatible}
    s(\psi^{-1}_{w,R}(u)) = c(1_G). 
\end{align}
\noindent 
Let $v\in B(w,R)\cap V(3R)$ for some $w \in W$. Then $ B(v,R) \subset V(2R)$. Let $g= \psi^{-1}_{w,R}(v) \in B_\Delta(R)$. We infer from \eqref{eq:Phi-map} and the $\Delta$-labeled graph isomorphism $\psi_{v, 2r}\circ \gamma_{g,B_\Delta (r)} \colon gB_\Delta(2r) \to B(v,2r)$ the following relation for all  $x \in A^{V(R)}$:   
\begin{align}
\label{eq:Phi-equal-sigma-s-locally-1}
    \Phi(x)\vert_{B(v, r)} \circ \psi_{v,r} \circ \gamma_{g,B_\Delta (r)} = f_{gB_\Delta (r), s\vert_{ g B_\Delta (r)}}^{+B_\Delta (r)}(x\vert_{B(v,2r)} \circ \psi_{v,2r}\circ \gamma_{g,B_\Delta (2r)}). 
\end{align}
\par 
\noindent 
Similarly, let $v\in V(3R) \setminus \overline{W}$ and  $x \in A^{V(R)}$. We infer from \eqref{eq:Phi-map}, \eqref{eq:definition-Phi-compatible}, and the $\Delta$-labeled graph isomorphism $\psi_{v, 2r} \colon B_\Delta(2r) \to B(v,2r)$ that   
\begin{align}
\label{eq:Phi-equal-sigma-s-locally-2}
   \Phi(x)\vert_{B(v, r)} \circ \psi_{v,r} =  f_{B_\Delta (r), s\vert_{B_\Delta (r)}}^{+B_\Delta (r)}(x\vert_{B(v,2r)} \circ \psi_{v,2r})
\end{align}
\par 
\noindent 
Let us denote
\[
Z = \Phi (A^{V(R)}) \subset A^{V(2R)}. 
\]
\par 
\noindent 
We  consider also the  
map $\Psi \colon Z \to A^{V(3R)}$ defined 
for each $z \in Z\subset  A^{V(2R)}$ and $v \in V(3R)$ by the formula: 
\begin{align}
\label{eq:Psi-map}
\Psi(z)(v) = 
\begin{cases}
     t(\psi^{-1}_{w,R}(v)) (x\vert_{B(v,r)}\circ \psi_{v,r}) & \text{ if } v\in B(w,R) \text{  for some } w \in W \\
d(1_G)\left(x\vert_{B(v,r)} \circ \psi_{v,r}\right) & \text{ if } v \in V(3R) \setminus \overline{W}. 
\end{cases}
\end{align}
\par 
\noindent
Because of the inclusions  $V(3R) \subset V(R) \subset V$, we obtain a canonical projection map    $\pi \colon A^{V(R)} \to A^{V(3R)}$ given by $ x\mapsto x\vert_{V{3R}}$ for every $x\in A^{V(R)}$. 
\\[6pt] 
\textbf{Claim:} we have  
\begin{equation}
\label{e:main-surjunctive-proof-1992}
    \Psi \circ \Phi = \pi.
\end{equation} 
Indeed, let $x \in A^{V(R)}$ and $v \in {V(3R)}$. Recall that  $M=N=B_\Delta(r)$ and $R=4r$ so that we have  
\[ 
B_\Delta(r)M = B_\Delta(r)N = B_\Delta(2r)=B_\Delta(R/2). 
\] 
On the other hand, we have by Lemma~\ref{l:sofic-B-V}.(i) that $B(v,R)\subset V(2R)$ and  $B(v,2R) \subset V(R)$ so that in particular $x\vert_{B(v,R)}$ is well-defined. We distinguish two cases according to whether $v \in \overline{W}$. Let us define 
\[
g= \begin{cases}
\psi^{-1}_{w,R}(v) \in B_\Delta(R) &\text{ if } v \in B(w,R)\cap V(3R) \text{ for some }w \in W \\
1_G & \text{ if } v \in  V(3R) \setminus \overline{W}.  
\end{cases}
\]
and fix some $z \in A^G$ extending $x\vert_{B(v,2R)} \circ  \psi_{v,2R} \circ \gamma_{g,B_\Delta (2R)} \in A^{gB_\Delta(2R)}$. 
\\[6pt] 
\textbf{Case 1:} $v \in B(w,R)\cap V(3R)$ for some $w \in W $. By \eqref{eq:Phi-equal-sigma-s-locally-1}, \eqref{eq:Psi-map}, and the relation $\sigma_t \circ \sigma_s=\Id$, we can thus compute: 
\begin{align*}
     \Psi ( \Phi (x))(v) 
     & =  
         t(g) (\Phi (x)\vert_{B(v,r)}\circ \psi_{v,r})
    &  \\ 
     &=   t(g) \left( \gamma_{g, B_\Delta (r)} \left(f_{gB_\Delta (r), s\vert_{ g B_\Delta (r)}}^{+B_\Delta (r)}\left(x\vert_{B(v,2r)} \circ \psi_{v,2r}\circ \gamma_{g, B_\Delta (2r)}\right) \right) \right) \\
     & = \sigma_t(\sigma_s(z)) (  g ) 
   &   \\ 
   & = z(g) &
   \\
    & = \left( x\vert_{B(v,2r)} \circ \psi_{v,2r} \circ \gamma_{g,B_\Delta (2r)}\right) (g) & \\ 
     & = x(\psi_{v,2r}(1_G))
    &  \\ 
     & = x(v). &
\end{align*}
\textbf{Case 2:}  $v \in V(3R)\setminus \overline{W}$. Similarly, we deduce from \eqref{eq:Phi-equal-sigma-s-locally-2} and \eqref{eq:Psi-map} and the relation $\sigma_t\circ \sigma_s=\Id$ that 
\begin{align*}
     \Psi ( \Phi (x))(v) 
     & =  
     d(1_G)\left(\Phi (x)\vert_{B(v,r)} \circ \psi_{v,r}\right)
    &  \\ 
       &=   d(1_G) \left( f_{B_\Delta (r), s\vert_{B_\Delta (r)}}^{+B_\Delta (r)}(x\vert_{B(v,2r)} \circ \psi_{v,2r})  \right) \\
     & = \sigma_t(\sigma_s(z)) (1_G) 
   &   \\ 
   & = z(1_G) &  
   \\
     & = \left( x\vert_{B(v,2r)} \circ \psi_{v,2r}\right) (1_G) & \\ 
     & = x(\psi_{v,2r}(1_G))
    &  \\ 
     & = x(v). &
\end{align*}
\par 
\noindent
Therefore, $\Psi(\Phi)(x)=\pi(x)$ for all $x \in A^{V(R)}$ and Claim \eqref{e:main-surjunctive-proof-1992} is proved. 
Consequently, we deduce  that 
\begin{equation}
\label{e: Z}
\Psi(Z) = \Psi(\Phi(A^{V(R)}))= \pi(A^{V(R)}) = A^{V(3R)}.
\end{equation}
\par
\noindent 
It follows that
\begin{equation}
\label{e:sofic-log(Z)}
\log \vert Z\vert \geq \log \vert A^{V(3R)}\vert= \log \vert A\vert^{\vert V(3R)\vert} =  \vert  V(3R)\vert \log \vert A \vert. 
\end{equation}
\par 
\noindent 
For all  $w \in W$, we have by the construction of $Z$ and $\Phi$ an isomorphism: 
\[
Z_{B(w,R)} \to \Gamma_{B_\Delta(R)}, \quad c \mapsto c \circ \psi_{w,R}.
\]
\par 
\noindent
Since $B_\Delta(r) \subset B_\Delta(R)$, we infer from \eqref{e:contrary}  that  $\Gamma_{B_\Delta(R)} \subsetneq A^{B_\Delta(R)}$ is a proper subset.  It follows that for all $w \in W$, we have: 
\begin{equation}
\label{e:z-sofic}
\vert Z_{B(w,R)} \vert = \vert\Gamma_{B_\Delta(R)}\vert \leq \vert A^{B_\Delta(R)}\vert -1.
\end{equation}
\par
\noindent 
Consequently, we find that:  
\begin{align*}
\log \vert Z\vert& \leq \log \bigl( \vert Z_{\overline{W}} \vert  \vert A^{V(2R) \setminus \overline{W}}\vert \bigr)
&  (\text{as } Z \subset Z_{\overline{W}} \times A^{V(2R)}) \\ 
& = \log \vert  Z_{\overline{W}}\vert  +
\log \vert A^{V(2R) \setminus \overline{W}}\vert 
&  \\ 
& \leq   \log \vert\prod_{w\in W} Z_{B(w,R)}\vert  + \log \vert A^{V(2R) \setminus \overline{W}}\vert 
&  (\text{as } Z_{\overline{W}} \subset \prod_{w\in W} Z_{B(w,R)}) \\ 
& =  \sum_{w\in W}\log \vert Z_{B(w,R)}\vert  + \vert V(2R) \setminus \overline{W}\vert \log \vert A \vert
&    \\
& \leq \vert W \vert  \log \bigl( \vert A \vert^{\vert B_\Delta(R) \vert}  - 1 \bigr) +  \vert V(2R) \setminus \overline{W}\vert  \log \vert A \vert
& (\text{by } \eqref{e:z-sofic}). 
\end{align*}
\par 
\noindent 
Therefore, we deduce that:
\begin{align*}
 \log \vert Z \vert & \leq 
  \left( \vert W \vert \vert B_\Delta(R) \vert  + \vert V(2R) \setminus \overline{W}\vert  \right)\log \vert A \vert +  \vert W \vert \log \left( 1- \vert A \vert^{-\vert B_\Delta(R) \vert} \right) 
\\
& = \vert V(2R) \vert \log \vert A \vert  +  \vert W \vert \log \left( 1- \vert A \vert^{-\vert B_\Delta(R) \vert} \right) 
\end{align*}
where the last equality follows from  \eqref{e:sofic-Vr}. 
Combining the above equality with the relation  \eqref{e:sofic-log(Z)}, we obtain:  
\[
\vert V(3R)\vert \log \vert A \vert  \leq \vert V(2R) \vert \log \vert A \vert  +  \vert W \vert \log \left( 1- \vert A \vert^{-\vert B_\Delta(R) \vert} \right).
\]
\par 
\noindent 
By dividing both sides by $\log \vert A \vert$ and  rearranging the terms,  we have: 
\begin{equation}
\label{e:sofic-V(Z)-2}
     \vert V(2R)\vert  \geq   \vert V(3R)\vert -  \vert W \vert \frac{\log \left( 1- \vert A \vert^{-\vert B_\Delta(R) \vert} \right)}{\log \vert A \vert}. 
\end{equation}
\par 
\noindent 
Since $\log \left( 1- \vert A \vert^{-\vert B_\Delta(R) \vert} \right) <0$ and $V(2R) \subset V$, we obtain     
\begin{align*}
\vert V \vert \geq \vert V(2R)\vert & \geq \vert V(3R)\vert  -  \vert W \vert \frac{\log \left( 1- \vert A \vert^{-\vert B_\Delta(R) \vert} \right)}{\log \vert A \vert} 
& (\text{by } \eqref{e:sofic-V(Z)-2})\\
& \geq  \vert V(3R)\vert  -  \frac{ |V(3R)|\log \left( 1- \vert A \vert^{-\vert B_\Delta(R) \vert} \right)}{ \vert B_\Delta(2R)\vert \log \vert A \vert} 
&  
 (\text{by } \eqref{eq:cardinality-w-main-A}) 
\\
& >  \vert V(3R)\vert (1 - \varepsilon)^{-1}. 
& (\text{by } \eqref{e:contra-hypothesis})
\end{align*}
\par 
\noindent 
It follows that 
$ 
 \vert V(3R)\vert < (1 - \varepsilon) \vert V \vert
$ 
which contradicts \eqref{e:sofic-V}.
Therefore, 
$\sigma_s$ must be  surjective. Since $\sigma_t \circ \sigma_s=\Id$, we conclude that $\sigma_s$ is invertible and the proof is thus complete. 
\end{proof}

\section{Invertibility of stably post-surjective NUCA over sofic groups with uniformly bounded singularity}
\label{s:dual-surjunctive-sofic}

By Lemma~\ref{l:equivalent-stable-invertible}, Theorem~\ref{t:main-C} reduces to the following generalization of the dual surjunctivity result \cite[Theorem~2]{kari-post-surjective}  for post-surjective CA to the  class of stably post-surjective NUCA over sofic group universes. 

\begin{theorem}
\label{t:post-sur-implies-pre-inj-first} 
Let $M$ be a finite subset of a  finitely generated sofic group $G$. Let $A$ be a finite set and let $S=A^{A^M}$. Let $s \in S^G$ be asymptotically constant such that $\sigma_s$ is stably post-surjective. Then $\sigma_s$ is invertible. 
\end{theorem}

\begin{proof}
Let $\Delta$ be a symmetric generating set of $G$, i.e., $\Delta=\Delta^{-1}$, such that $1_G \in \Delta$. Recall the balls $B_\Delta(n) \subset G$ of radius $n \geq 0$ in the corresponding Cayley graph $C_\Delta(G)$ of $G$. 
By hypothesis, we can find a constant configuration $c \in S^G$ and a finite subset $F\subset G$   such that $s\vert_{G \setminus F} =  c\vert_{G \setminus F}$.   Since $\sigma_s$ is stably post-surjective, we infer from  \cite[Lemma~13.2]{phung-tcs}  a finite subset $E \subset G$  such that for all $g \in G$ and  $x, y\in A^G$ with $y\vert_{G \setminus \{g\}} =\sigma_s(x)\vert_{G \setminus \{g\}}$, there exists $z \in A^G$ such that $\sigma_s(z)=y$ and 
$z\vert_{G \setminus gE}= x\vert_{G \setminus gE}$. 
We claim that $\sigma_s$ is  
pre-injective. Indeed, 
 there exists on the contrary distinct asymptotic configurations $p, q \in A^G$ with  $\sigma_s(p)= \sigma_s(q)$. 
Up to enlarging $M$ and $F$, we can suppose that  $p\vert_{FM} \neq q\vert_{FM}$. 
We fix $r \in \N$ such that $r \geq 1$ and   
$B_\Delta(r) \supset M \cup E\cup F$.   
Up to enlarging $M$, $E$, and $F$ again if necessary,  we can suppose  that 
$$M=E=F=B_\Delta(r).$$
We infer from \eqref{e:induced-local-maps},  \eqref{e:induced-local-maps-general}, and the relation $\sigma_s(p)= \sigma_s(q)$  
that: 
\begin{align} 
\label{e:mutually-erasable-pattern}
   f_{F,s\vert_F}^{+M}(p\vert_{FM}) = f_{F,s\vert_F}^{+M}(q\vert_{FM}).
\end{align}
\par 
\noindent 
Let $R=4r$. The theorem is trivial when $\vert A \vert \leq 1$ so we suppose from now on that $\vert A \vert \geq 2$.  Let us fix $0<\varepsilon< \frac{1}{2}$ small enough
  so that 
  \begin{equation}
  \label{e:post-surjective-var-epsilon-3-1}
    \vert A \vert^{\varepsilon} \left( 1 - \vert A \vert^{-|B_\Delta(R)|} \right)^{\frac{1}{2|B_\Delta(2R)|}}
    < 1,
  \end{equation}
\par 
\noindent 
By Theorem~\ref{t:sofic-character}, there exists a finite $S$-labeled graph $\GG=(V,E)$ associated to the pair $(3R, \varepsilon)$ 
such that
\begin{equation} 
\label{e:sofic-V-2post-surjective}
    \vert V(3R) \vert \geq (1 - \varepsilon) \vert V \vert \geq (1 - \varepsilon)\vert V(R) \vert
\end{equation}
where
  $V(n)$ consists   of $v \in V$ such that there exists a $\Delta$-labeled graph
isomorphism $\psi_{v,n}\colon  B(v,n) \to B_\Delta(n)$ mapping 
 $v$ to $1_G$ and $B(v,n)= B_{\GG}(v,n)$. 
\par
We recall briefly the map $\Phi \colon A^{V(R)} \to A^{V(3R)}$ constructed in the proof of Theorem~\ref{t:main-A-proof}. 
From Lemma~\ref{l:sofic-B-V}.(ii), we have a subset $W \subset V(3R) $ such that the balls 
$B(w,R)$, for all $w\in W$,  are pairwise disjoint  and $V(3R) \subset \bigcup_{w \in W}B(w,2R)$. In particular,  
\begin{equation}
\label{e:v'-and-v-post-surjective-3}
    \vert V(3R) \vert \leq \vert W \vert \vert B_\Delta(2R) \vert.
\end{equation}
\par 
Let 
$\overline{W} = \coprod_{w \in W} B(w,R) \subset V(2R)$. 
For $x \in A^{V(R)}$ and $v \in V(3R)$, let 
\begin{align}
\label{eq:Phi-map-dual}
\Phi(x)(v) 
=
\begin{cases}
    s(\psi^{-1}_{w,R}(v)) (x\vert_{B(v,r)}\circ \psi_{v,r}) & \text{ if } v\in B(w,R) \text{  for some } w \in W\\
    c(1_G)\left(x\vert_{B(v,r)} \circ \psi_{v,r}\right) & \text{ if } v \in V(3R) \setminus \overline{W}.  
\end{cases}
\end{align}
\par 
\noindent
For $v\in B(w,R)\cap V(3R)$ where $w \in W$, we have for  $g= \psi^{-1}_{w,R}(v) \in B_\Delta(R)$ and for all  $x \in A^{V(R)}$ that: 
\begin{align}
\label{eq:Phi-equal-sigma-s-locally-1-dual}
    \Phi(x)\vert_{B(v, r)} \circ \psi_{v,r} \circ \gamma_{g,B_\Delta (r)} = f_{gB_\Delta (r), s\vert_{ g B_\Delta (r)}}^{+B_\Delta (r)}(x\vert_{B(v,2r)} \circ \psi_{v,2r}\circ \gamma_{g,B_\Delta (2r)}). 
\end{align}
\par 
\noindent
For $v\in V(3R) \setminus \overline{W}$ and  $x \in A^{V(R)}$, we have    
\begin{align}
\label{eq:Phi-equal-sigma-s-locally-2-dual}
   \Phi(x)\vert_{B(v, r)} \circ \psi_{v,r} =  f_{B_\Delta (r), s\vert_{B_\Delta (r)}}^{+B_\Delta (r)}(x\vert_{B(v,2r)} \circ \psi_{v,2r})
\end{align}
 
\par 
\noindent
Let $y \in A^{V(3R)}$ and fix an arbitrary $x \in A^{V(R)}$. Let $v_1,..., v_n$ be elements of $V(3R)$. 
We describe inductively below a procedure on $1\leq i\leq n$ which allows us to obtain a sequence $x_0=x, x_1,  ..., x_n$ such that
 $$\Phi(x_i)\vert_{\{v_1, ..., v_i\}}=y\vert_{\{ v_1, ..., v_i\}}, \quad  i=1,..., n.$$ 
 For $i=1,..., n$, let $g_i= \psi^{-1}_{w,R}(v_i)$ if $v_i \in B(w,R)\cap V(3R)$ for some $w \in W$ and $g_i= 1_G$ otherwise. 
By \eqref{eq:Phi-equal-sigma-s-locally-1-dual}, \eqref{eq:Phi-equal-sigma-s-locally-2-dual}, and the choice of $E$,   we can modify, via  the isomorphism $\psi_{v, r}\circ \gamma_{g_i, B_\Delta(r)}$,  the restriction  $x_{i-1}\vert_{B(v_{i}, r)}$ to obtain a new configuration $x_{i}\in A^{V(R)}$ such that  $x_{i}\vert_{V(R)\setminus B(v_{i}, r)}= x_{i-1}\vert_{V(R)\setminus B(v_{i}, r)}$ and 
$$\Phi(x_{i})\vert_{V(3R)\setminus \{v_{i}\}}= \Phi(x_{i-1})\vert_{V(3R)\setminus \{v_{i}\}}, \quad \quad \Phi(x_{i})(v_{i})= y(v_{i}). 
$$
Consequently, $\Phi(x_n)= y$ and we deduce that   $\Phi$ is surjective. In particular,    
  \begin{equation} 
  \label{e:surj-post-phi-surjective-v-3r}
  \vert \Phi(A^{V(R)})\vert = \vert A^{V(3R)}\vert .
  \end{equation}
  \par 
  \noindent 
By  \eqref{e:mutually-erasable-pattern},  we have    
$f_{B_\Delta(r),s\vert_{B_\Delta(r)}}^{+B_\Delta (r)}(p\vert_{B_\Delta(2r)}) = f_{B_\Delta(r),s\vert_{B_\Delta(r)}}^{+B_\Delta (r)}(q\vert_{B_\Delta(2r)})$. Let 
$$\Lambda = \prod_{w \in W}\psi_{w,R}^{-1}(\{p\vert_{B_\Delta(R)}, q\vert_{B_\Delta(R)}\}) \subset A^{\overline{W}}.$$ 
Then for 
every $e \in A^{V(R) \setminus {\overline{W}}}$ and all $x,y \in \Lambda$, we deduce from \eqref{eq:Phi-map-dual} and $R=4r$ that $ 
    \Phi(e \times x) =  \Phi(e \times y)$. 
Consequently,  
\[
\Phi(A^{V(R)}) = \Phi\left(A^{V(R) \setminus {\overline{W}}}\times \prod_{w \in W} \left( A^{B(w,R)}\setminus \psi_{w,R}^{-1}(q\vert_{B_\Delta(R)})\right)\right). 
\]
Therefore,  as $|B(w,R)|=|B_\Delta(R)|$ for all $w \in W$, we obtain the estimation 
\begin{align}
\label{e:surj-post-phi-surjective-v-3r-c}
     \vert \Phi(A^{V(R)})\vert 
     & \leq \left|A^{V(R) \setminus {\overline{W}}}\times \prod_{w \in W} \left( A^{B(w,R)}\setminus \psi_{w,R}^{-1}(q\vert_{B_\Delta(R)})\right)\right|\\
     & \leq \vert A \vert^{\vert V(R) \setminus {\overline{W}}\vert} (\vert A \vert^{\vert B_\Delta(R)\vert}-1)^{\vert W\vert}. \nonumber 
\end{align}
Finally, we find that 
  \begin{align*}
    \vert A\vert^{\vert {V(3R)}\vert } &  =  \vert \Phi(A^{V(R)})\vert  & \text{(by } \eqref{e:surj-post-phi-surjective-v-3r}) \\ 
    & \leq 
   \vert A \vert^{\vert V(R) \vert - \vert{\overline{W}}\vert} (\vert A \vert^{\vert B_\Delta(R)\vert}-1)^{\vert W\vert}  
    & \text{(by } \eqref{e:surj-post-phi-surjective-v-3r-c}) 
    \\
    & = 
   \vert A \vert^{\vert V(R) \vert} (1 - \vert A \vert^{- \vert B_\Delta(R)\vert})^{\vert W\vert}  
    &  \text{(as } |\overline{W}|=|B_\Delta(R)| |W|) 
    \\ 
    & \leq 
     \vert A \vert^{\vert V(R) \vert} (1 - \vert A \vert^{- \vert B_\Delta(R)\vert})^{\frac{\vert V(3R)\vert}{  \vert B_\Delta(2R) \vert}} & \text{(by }\eqref{e:v'-and-v-post-surjective-3}) \\
      &  \leq 
     \vert A \vert^{\vert V(R) \vert} (1 - \vert A \vert^{- \vert B_\Delta(R)\vert})^{\frac{\vert V(R)\vert}{2  \vert B_\Delta(2R) \vert}}
      & (\text{by }\eqref{e:sofic-V-2post-surjective})
      \\
      &  < 
     \vert A \vert^{\vert V(R) \vert}  \vert A \vert^{- \varepsilon \vert V(R) \vert}
      & (\text{by }\eqref{e:post-surjective-var-epsilon-3-1}) \\
      & \leq \vert A \vert^{\vert V(3R)\vert} & \text{(by } \eqref{e:sofic-V-2post-surjective}) 
\end{align*} 
  which is a contradiction. Hence, $\sigma_s$ must be pre-injective. Since $\sigma_s$ is also stably post-surjective by hypothesis, we conclude from \cite[Theorem~13.4]{phung-tcs} that  $\sigma_s$ is invertible and the proof is complete. 
\end{proof}

We can generalize in a direct manner the above proof of Theorem~\ref{t:post-sur-implies-pre-inj-first} to obtain  Theorem~\ref{t:main-D} as follows. 

\begin{proof}[Proof of Theorem~\ref{t:main-D}] 
Let $\Delta$ be a symmetric generating set of $G$, i.e., $\Delta=\Delta^{-1}$, such that $1_G \in \Delta$.  As $\sigma_s$ is stably post-surjective, \cite[Lemma~13.2]{phung-tcs} provides a finite subset $E \subset G$  such that for all $g \in G$ and  $x, y\in A^G$ with $y\vert_{G \setminus \{g\}} =\sigma_s(x)\vert_{G \setminus \{g\}}$, there exists $z \in A^G$ such that $\sigma_s(z)=y$ and 
$z\vert_{G \setminus gE}= x\vert_{G \setminus gE}$. 
As in the proof of Theorem~\ref{t:main-D}, we only need to show that $\sigma_s$ is  
pre-injective. Suppose on the contrary that  there exists a finite subset $K \subset G$ and configurations $p, q \in A^G$ such that  $\sigma_s(p)= \sigma_s(q)$, $p\vert_{G \setminus K}= q\vert_{G \setminus K}$, and $p\vert_{K} \neq q\vert_{K}$.    
Up to enlarging $M$, $E$, and $K$ if necessary,  we can suppose without loss of generality that $1\in M=M^{-1}$ and $M=E=K$.
Since $s$ has uniformly bounded singularity, we can find a constant configuration $c \in S^G$ and finite subset $H \subset G$  such that $s\vert_{HM \setminus H}=c\vert_{HM \setminus H}$ and $M \subset H$. 
We fix $r \geq 1$  large enough so that    
$H \subset B_\Delta(r)$. Let $R=4r$ and we define an asymptotically constant configuration $t\in S^G$ by 
$$t\vert_{B_\Delta(r)}=s\vert_{B_\Delta(r)}, \quad t\vert_{G\setminus B_\Delta(r)} = c\vert_{G \setminus B_\Delta(r)}. 
$$
Up to enlarging $M,E,K$ again, we can assume that 
$M=E=K= B_\Delta(r)$ and let $F=M$. 
From this point, it suffices to apply the rest of the proof of Theorem~\ref{t:post-sur-implies-pre-inj-first} for $t$ instead of $s$ to obtain a contradiction. The proof of Theorem~\ref{t:main-D} is thus complete. 
\end{proof}

\section{Stable surjunctivity of NUCA over sofic groups  with uniformly bounded singularity} 
\label{s:surjunctivity-sofic-general} 
We prove the following key technical lemma which is useful for the proof of Theorem~\ref{t:main-F} and allows us to reduce Theorem~\ref{t:main-B} to Theorem~\ref{t:main-A}. 
\begin{lemma}
\label{l:main-lemma-singular}
Let $A$ be an alphabet and let $G$ be a group. Let $M\subset G$ be a finite subset and $S=A^{A^M}$. Suppose that $\sigma_t \circ \sigma_s=\Id$ for some $s,t \in S^G$ and $s$ has uniformly bounded singularity. Then 
for each $E \subset G$ finite, there exists asymptotically constant configurations $p,q\in S^G$ such that $p\vert_E = s\vert_E$, $q \vert_E= t\vert_E$, and $\sigma_q \circ \sigma_p=\Id$. 
\end{lemma}

\begin{proof}
Up to enlarging $M$, we can clearly assume that $1_G \in  M=M^{-1}$. 
Let $E\subset G$ be a finite subset. Up to enlarging $E$, we can suppose without loss of generality that $ M  \subset E$ and $E=E^{-1}$.  As $s$ has uniformly bounded singularity, there exists a constant configuration $c \in A^{A^M}$  and a finite subset $F\subset G$ containing $E^3$ such that $s\vert_{FE^3\setminus F}= c\vert_{FE^3\setminus F}$. We define an asymptotically constant configuration $p\in S^G$ by setting $p\vert_{FE}=s\vert_{FE}$ and $p\vert_{G\setminus FE}=c\vert_{G\setminus FE}$. 
We fix $g_0 \in FE^2 \setminus FE$ and 
  define $q \in S^G$ by  $q(g)=t(g_0)$ if $g \in G\setminus FE$ and $q(g)=t(g)$ if $g \in FE$. Then $q$ is asymptotic to the constant configuration $d \in S^G$ defined by $d(g)= t(g_0)$ for all $g \in G$. Since $1_G \in M$, we have a projection $\pi \colon A^{M^2}\to A$ given by $x\mapsto x(1_G)$. Since $E\subset E^3 \subset F \subset FE$, we deduce from our construction that $p\vert_E =s\vert_E$ and $q\vert_E =t\vert_E$. To conclude, we only need to check that $\sigma_q\circ \sigma_p=\Id$.
  \par 
  Let $g \in FE^2 \setminus FE$. Then $gE \subset FE^3\setminus F$ since $E=E^{-1}$.  Consequently,  
$s\vert_{gE}= c\vert_{gE}=p\vert_{gE}$ and thus $s\vert_{gM}= c\vert_{gM}=p\vert_{gM}$   since $M \subset E$. Hence, the condition $\sigma_t( \sigma_s(x))(g)=x(g)$ for all $x \in A^G$ is equivalent to $t(g) \circ f^{+M}_{M, c\vert_M} = \pi$ by Lemma~\ref{l:compo-id}. Similarly,  the condition $\sigma_q( \sigma_p(x))(g)=x(g)$ for all $x \in A^G$  amounts to $q(g) \circ f^{+M}_{M, c\vert_M} = \pi $. Since $q(g)=t(g_0)$ and $\sigma_t\circ \sigma_s=\Id$, we conclude from the above discussion that $\sigma_q( \sigma_p(x))(g)=x(g)$ for all $x \in A^G$. 
 \par 
 Let  $g \in  FE$. 
 Since $s\vert_{FE^3\setminus F}= c\vert_{FE^3\setminus F}$,   $p\vert_{FE}=s\vert_{FE}$, and $p\vert_{G\setminus FE}=c\vert_{G\setminus FE}$ by construction, we have  $p\vert_{FE^3}=s\vert_{FE^3}$. In particular, $p\vert_{gM}=s\vert_{gM}$ since $gM \subset (FE)E=FE^2 \subset FE^3$.  Therefore, we can infer from the relations  $\sigma_t\circ \sigma_s=\Id$ and  $q(g)=t(g)$  that $\sigma_q( \sigma_p(x))(g)=x(g)$ for all $x \in A^G$.
 \par 
 Finally, let $g \in G \setminus FE^2$. Since $p\vert_{G\setminus FE}=c\vert_{G\setminus FE}$,  $q\vert_{G\setminus FE}=d\vert_{G\setminus FE}$, and since $c,d$ are constant, we deduce that $q(g)= d(g_0)=t(g_0)$ and $p\vert_{gM}= c\vert_{gM}$. 
 The condition $\sigma_q( \sigma_p(x))(g)=x(g)$ for all $x \in A^G$ is thus  equivalent to $t(g_0) \circ f^{+M}_{M, c\vert_M} = \pi$ by Lemma~\ref{l:compo-id}. But since $\sigma_t\circ\sigma_s=\Id$ and $s\vert_{g_0M}=c\vert_{g_0M}$, another application of  Lemma~\ref{l:compo-id} shows that $t(g_0) \circ f^{+M}_{M, c\vert_M} = \pi$. 
 Thus, $\sigma_t( \sigma_s(x))(g)=x(g)$ for all $x \in A^G$. 
Therefore, we conclude that $\sigma_q\circ \sigma_p=\Id$ and the proof is complete.  
\end{proof}

We are now in position to prove Theorem~\ref{t:main-B}. 

\begin{proof}[Proof of Theorem~\ref{t:main-B}] 
    Since $\sigma_s$ is stably injective by hypothesis, we deduce from \cite[Theorem~A]{phung-tcs} that there exists a finite subset $N \subset G$ and $t \in T^G$, where $T=A^{A^N}$, such that $\sigma_t\circ \sigma_s= \Id$. In particular, it suffices to show that $\sigma_s$ is surjective for $\sigma_s$ to be invertible. Up to enlarging $M$ and $N$, we can suppose that $M=N$ and thus $S=T$. We suppose on the contrary that $\sigma_s$ is not surjective. Since $\Gamma=\sigma_s(A^G)$ is closed in $A^G$ with respect to the prodiscrete topology by \cite[Theorem~4.4]{phung-tcs}, there must exists a finite subset $E \subset G$ such that $\Gamma_{E} \subsetneq A^E$. 
    Since $s$ has uniformly bounded singularity, we infer from Lemma~\ref{l:main-lemma-singular} that there exists asymptotically constant configurations $p,q\in S^G$ such that $p\vert_{EM} = s\vert_{EM}$, $q \vert_{EM}= t\vert_{EM}$, and $\sigma_q \circ \sigma_p=\Id$. Let $\Lambda=\sigma_p(A^G)$ then it follows from $p\vert_{EM} = s\vert_{EM}$ that $\Lambda_E=\Gamma_E\subsetneq A^E$. We deduce that $\sigma_p$ is not surjective. In particular, $\sigma_p$ is not invertible. On the other hand, the condition $\sigma_q \circ \sigma_p=\Id$ shows that $\sigma_q$ is stably injective by \cite[Theorem~A]{phung-tcs}. We can thus apply Theorem~\ref{t:main-A} to deduce that $\sigma_p$ is invertible. Therefore, we obtain a contradiction and the proof is complete. 
\end{proof}


\bibliographystyle{siam}

\end{document}